%% file: 13-12-19_ModAss.tex
\documentclass[12pt,a4paper]{article}
\usepackage{amssymb, amsmath, amsthm, yhmath}
\usepackage{tikz}
\usetikzlibrary{matrix,arrows}
\usepackage{hyperref}

\input{commands.tex}

\newcommand{\cyc}[1]{\left(\hspace{-.6ex}\left( #1 \right)\hspace{-.6ex}\right)}	
\newcommand{\cc}{\mathbf{c}}	
\newcommand{\oQO}{\mathcal{QO}}	
\newcommand{\oQC}{\mathcal{QC}}	
\newcommand{\oMod}[1]{\textbf{Mod}\left(#1\right)}	
\newcommand{\cM}{\mathsf{M}}

\newcommand{\oP}{\mathcal{P}}

\input{ModAss_pic.tex}

\input{environments.tex}

\title{The modular envelope of the cyclic operad $\oAss$}
\author{
Martin Doubek\thanks{The author was supported by GA\v{C}R 201/12/G028 at the initial stage of the project and by GA\v{C}R P201/13/27340P at the later stage.},  \\
Charles University, Prague\\
\texttt{martindoubek@seznam.cz}
}
\date{\today}

\begin{document}

\maketitle
\abstract{
We give a direct combinatorial proof that the modular envelope of the cyclic operad $\oAss$ is the modular operad of (the homeomorphism classes of) $2D$ compact surfaces with boundary with marked points.
}

\section{Introduction}

The forgetful functor from modular to cyclic operads has a left adjoint $\oMod{}$ called modular envelope.
The modular envelope $\oMod{\oCom}$ of the cyclic operad $\oCom$ has been described by M. Markl in \cite{MarklLoop}, where it's relevance in closed string field theory was explained.
$\oMod{\oCom}$ is isomorphic to the modular operad $\oQC$, which consists of homeomorphism classes of $2D$ compact surfaces with boundary (equivalently with punctures) and the operadic composition on $\oQC$ is given by gluing of boundary components (punctures).
In \cite{QOC}, together with coauthors, we studied the operad $\oQO$, which consists of homeomorphism classes of $2D$ compact surfaces with boundary and ``open string ends'' on the boundary (equivalently marked points on the boundary).
Operad $\oQO$ plays a role in the open string field theory analogous to that of $\oQC$ in the closed theory.
The genus zero part of $\oQO$ is isomorphic to the cyclic operad $\oAss$ and it is therefore natural to expect $\oQO \cong \oMod{\oAss}$.
Unlike for $\oQC$, this is not obvious.
In this note, we provide an elementary proof.

Let us note that \emph{derived} modular envelope of $\oAss$ has been studied in the topological context by J. Giansiracusa in \cite{Gia} and also by K. Costello in \cite{Cos}.

We thank Bra\v no Jur\v co, Martin Markl and Korbinian M\"unster for helpful discussions.

%


\section{Reminder}

We recall the notion of modular operad, since later we will be referring to the axioms:

\begin{definition} \label{DEFModOp}
Modular operad $\oP$ in a symmetric monoidal category $\cM$ consists of a collection 
$$\set{\oP(C,G)}{G\in\N_0,\ C\textrm{ a finite set}}$$ of objects of $\cM$ and collections
\begin{gather*}
	\set{\oP(\rho):\oP(C,G)\to\oP(D,G)}{\rho:C\to D\textrm{ a bijection}}, \\
	\set{\ooo{a}{b}:\oP(C_1\!\sqcup\!\{a\},G_1)\ot\oP(C_2\!\sqcup\!\{b\},G_2)\!\to\!\oP(C_1\!\sqcup\!C_2,G_1\!+\!G_2)}{G_1,G_2\!\in\!\N_0,\ C_1,C_2\textrm{ finite sets}}, \\
	\set{\oxi{ab}:\oP(C\sqcup\{a,b\},G)\to\oP(C,G+1)}{G\in\N_0,\ C\textrm{ a finite set}}.
\end{gather*}
of morphisms of $\cM$.
These data are required to satisfy the following axioms:
\begin{enumerate}
\item $\ooo{a}{b}  = \ooo{b}{a} \tau$, where $\tau:A\ot B\to B\ot A$ is the monoidal symmetry,
\item $\oP(\id_C)=\id_{\oP(C)},\ \oP(\rho\sigma)=\oP(\rho) \ \oP(\sigma)$ for any composable pair of bijections $\rho,\sigma$,
\item $(\oP(\rho|_{C_1}\sqcup\sigma|_{C_2})) \ \ooo{a}{b} = \ooo{\rho(a)}{\sigma(b)} \ (\oP(\rho)\ot\oP(\sigma))$
\item $\oP(\rho|_{C}) \ \oxi{ab} = \oxi{\rho(a)\rho(b)}\oP(\rho)$
\item $\oxi{ab} \ \oxi{cd} = \oxi{cd} \ \oxi{ab}$
\item $\oxi{ab} \ \ooo{c}{d} = \oxi{cd} \ \ooo{a}{b}$
\item $\ooo{a}{b} \ (\oxi{cd}\ot\id) = \oxi{cd} \ \ooo{a}{b}$
\item $\ooo{a}{b} \ (\id\ot\ooo{c}{d}) = \ooo{c}{d} \ (\ooo{a}{b}\ot\id)$
\end{enumerate}
whenever the expressions make sense.

We emphasize that $\oxi{ab}=\oxi{ba}$ by the definition.
$3.,4.$ express the equivariance of $\ooo{}{},\oxi{}$ respectively and we call $5.-8.$ the associativity axioms.
\end{definition}

\begin{remark}
Precomposing both sides of $7.$ with $\tau$, we obtain a mirror image of $7.$:
\begin{enumerate}
\item[$7'.$] $\ooo{a}{b} \ (\id\ot\oxi{cd}) = \oxi{cd} \ \ooo{a}{b}$
\end{enumerate}
\end{remark}

We require the reader to read Section V.A of \cite{QOC} (pages 23--25) to familiarize herself with the definition of the operad $\oQO$ and notation introduced there.
We use it freely in this paper.
For $\{\cc_1,\cc_2,\ldots,\cc_b\}^g\in\oQO(\bigcup_{i=1}^b\cc_i,G=2g+b-1)$, $g$ is called geometric genus.

The cyclic operad $\oAss$ is the $G=0$ part of $\oQO$, thus there is a canonical inclusion $\oAss\to\oQO$.
The elements of $\oAss$ are of the form $\{\cyc{C}\}^0$, which we denote simply by $\cyc{C}$.

\begin{remark}
Most naturally, the operads $\oAss$ and $\oQO$ live in the category of sets.
However, since we considered Feynman transform (an analogue of bar construction for modular operads) of $\oAss$ and $\oQO$ in \cite{QOC}, it is useful to consider their linear spans or even view them as chain complexes concentrated in zero degree.
The rest of the article is independent of this choice of the underlying category $\cM$.
\end{remark}


\section{The result}

We prove:

\begin{theorem} \label{THMTheResult}
Let $\cM$ be any of the categories of sets, (graded) vector spaces\footnote{With degree $0$ linear maps.} or chain complexes.
For any modular operad $\oP$ in $\cM$ and any cyclic operad morphism $f:\oAss\to\oP$, there is a unique modular operad morphism $\tilde{f}:\oQO\to\oP$ such that the following diagram commutes:
$$\begin{tikzpicture}
\matrix (m) [matrix of math nodes, row sep=3em, column sep=2.5em, text height=1.5ex, text depth=0.25ex]
{ \oAss & & \oQO \\
& & \oP \\ };
\path[->,font=\scriptsize] (m-1-1) edge (m-1-3);
\path[->,font=\scriptsize] (m-1-1) edge node[below left]{$\forall f$} (m-2-3);
\path[->,dashed,font=\scriptsize] (m-1-3) edge node[right]{$\exists!\tilde{f}$} (m-2-3);
\end{tikzpicture}$$
The horizontal map is the canonical inclusion.

Thus, by a standard general nonsense argument, $\oQO$ is isomorphic to $\oMod{\oAss}$ as a modular operad.
\end{theorem}

The notation of Theorem \ref{THMTheResult} will be fixed thorough the paper.

The following lemma is a statement about gluing ($\oxi{}$'s) of surfaces (elements of $\oQO$) out of elementary parts (elements of $\oAss$):

\begin{lemma}
Let 
$$q=\{\cc_1,\cc_2,\ldots,\cc_b\}^g$$
be an element of $\oQO(\bigcup_{i=1}^b\cc_i,G=2g+b-1)$.
For each cycle $\cc_i$, choose its representing $|\cc_i|$-tuple $C_i$.
Choose an arbitrary distinct elements $e_1,\ldots,e_{2G}$ s.t. $e_j\not\in\bigcup_{i=1}^b\cc_i$ for each $j$.
Define $a_q \in \oAss(\bigcup_{i=1}^b\cc_i \sqcup \{e_1,\ldots,e_{2G}\})$ by
\begin{gather} \label{EQaq}
a_q:=\cyc{C_1 e_1 C_2 e_2 e_3 C_3 e_4 \cdots e_{2b-3} C_b e_{2b-2} e_{2b-1} e_{2b} \cdots e_{4g+2b-2}}.
\end{gather}
$C_1e_1C_2\cdots$ is the usual concatenation of tuples viewing $e_i$ as a $1$-tuple $(e_i)$.
Then 
\begin{gather} \label{EQCanForm}
q = \oxi{12}\oxi{34}\cdots\oxi{2b-3,2b-2}\oxi{2b-1,2b+1}\oxi{2b,2b+2}\cdots\oxi{2b+4g-5,2b+4g-3}\oxi{2b+4g-4,2b+4g-2}(a_q)
\end{gather}
$\oxi{ij}$ is a shorthand for $\oxi{e_ie_j}$.
\end{lemma}

\begin{proof}
Every pair of the last $2g$ $\xi$'s in \eqref{EQCanForm} removes four consecutive $e_i$'s and increases the geometric genus:
\begin{gather*}
\oxi{2G-2,2G}(a_q) = \{\cyc{C_1 e_1 C_2 e_2 e_3 C_3 e_4 \cdots e_{2b-3} C_b e_{2b-2} e_{2b-1} e_{2b} \cdots e_{2G-4} e_{2G-3}} \cyc{e_{2G-1}}\}^0 \\
\oxi{2G-1,2G-3}\oxi{2G-2,2G}(a_q) = \{\cyc{C_1 e_1 C_2 e_2 e_3 C_3 e_4 \cdots e_{2b-3} C_b e_{2b-2} e_{2b-1} e_{2b} \cdots e_{2G-4}}\}^1
\end{gather*}
Thus
$$\oxi{2b-1,2b+1}\oxi{2b,2b+2}\cdots\oxi{2G-3,2G-1}\oxi{2G-2,2G}(a_q) = \{\cyc{C_1 e_1 C_2 e_2 e_3 C_3 e_4 \cdots e_{2b-3} C_b e_{2b-2}}\}^g =: E.$$
Next, each of the first $b$ $\xi$'s in \eqref{EQCanForm} removes a pair of $e_i$'s and separates a cycle:
$$\oxi{2b-3,2b-2}(E) = \{\cyc{C_1 e_1 C_2 e_2 e_3 C_3 e_4 \cdots e_{2b-4}} \cyc{C_b}\}^g$$
and thus
$$\oxi{12}\oxi{34}\cdots\oxi{2b-3,2b-2}(E)=\{\cyc{C_1} \cyc{C_2} \cdots \cyc{C_b}\}^g = q.$$
\end{proof}

The expression \eqref{EQCanForm} is called a canonical expression of $q$.
There are many canonical expressions of $q$:
\begin{enumerate}
\item $e_i$'s are not uniquely determined,
\item given $\cc_i$, it's representing tuple $C_i$ is not uniquely determined,
\item the order of $C_i$'s in \eqref{EQaq} is not uniquely determined.
\end{enumerate}

By definition, any morphism of modular operads commute with $\oxi{}$'s, hence:

\begin{lemma} \label{LEMCanFormg}
\begin{gather} \label{EQCanFormg}
\tilde{f}(q) = \oxi{12}\oxi{34}\cdots\oxi{2b-3,2b-2}\oxi{2b-1,2b+1}\cdots\oxi{2b+4g-4,2b+4g-2} f(a_q),
\end{gather}
hence the morphism $\tilde{f}$ of modular operads is uniquely determined by the morphism $f$ of cyclic operads. \qed
\end{lemma}

The expression \eqref{EQCanFormg} is called a canonical expression of $\tilde{f}(q)$ iff omitting $f$ in \eqref{EQCanFormg} yields a canonical expression of $q$.

To be able to effectively calculate with the canonical expressions, we introduce the following pictorial notation:
For example, 
\begin{gather} \label{EQExampleThreeXis}
\oxi{25}\oxi{38}\oxi{67}f\cyc{12\cdots 8}
\end{gather}
is represented by the picture
$$\PICAAA$$
The labels correspond to $\cyc{12\cdots 8}$ and for each $\oxi{ij}$ in front of $f\cyc{12\cdots 8}$, an arc connects $i$ to $j$.
To further simplify, we discard the labels at endpoints of arcs, thus obtaining
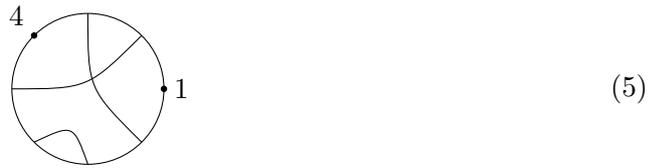
\begin{gather} \label{PICEx}
\PICAAB
\end{gather}

Are we able able to reconstruct the expression \eqref{EQExampleThreeXis} from the picture \eqref{PICEx}?
Not exactly: first, we have to choose labels of the endpoints of arcs.
Second, we have to choose an order of $\oxi{}$'s, i.e. order of the arcs.
But both these choices are irrelevant in the following sense:

\begin{lemma} \label{LEMPicturesOK}
Let $a_q=\cyc{\cdots e_1\cdots e_2\cdots}$ be as above and let $a'_q$ be obtained by replacing each $e_i$ by $e'_i$.
Then 
$$\oxi{12}\oxi{34}\cdots f(a_q) = \oxi{1'2'}\oxi{3'4'}\cdots f(a'_q)$$
and if the LHS is a canonical expression of $\tilde{f}(q)$, then so is the RHS.

Moreover, let $\sigma$ be a permutation of pairs $\{e_1,e_2\},\{e_3,e_4\},\ldots$.
Then 
$$\oxi{12}\oxi{34}\cdots f(a_q) = \oxi{\sigma(12)}\oxi{\sigma(34)}\cdots f(a_q)$$
and the RHS is a canonical expression of $\tilde{f}(q)$.
\end{lemma}

\begin{proof}
For the first part, $a_q\in\oAss(C\sqcup\{e_1,e_2,\ldots,e_{2G}\})$ and let's define a bijection $\rho:C\sqcup\{e_1,e_2,\ldots\}\to C\sqcup\{e'_1,e'_2,\ldots\}$ by $\rho(e_i):=e'_i$ for each $i$ and $\rho|_C=\id_C$.
Then $a'_q=\oAss(\rho)(a_q)$.
Since $f$ is a morphism of cyclic operads, by a repeated use of the equivariance of $\oxi{}$ we get
\begin{gather*}
\oxi{1'2'}\oxi{3'4'}\cdots\oxi{(2G-1)',(2G)'} f(a'_q) = \oxi{1'2'}\oxi{3'4'}\cdots\oxi{(2G-1)',(2G)'}\oP(\rho)f(a_q) = \\
= \oxi{1'2'}\oxi{3'4'}\cdots\oxi{(2G-1)',(2G)'}\oP(\rho|_{C\sqcup\{e_1,e_2,\ldots,e_{2G-3},e_{2G-2}\}})\oxi{2G-1,2G} f(a_q) = \cdots \\
\cdots = \oP(\rho|_{C}) \oxi{12}\oxi{34}\cdots \oxi{2G-1,2G} f(a_q) = \oxi{12}\oxi{34}\cdots \oxi{2G-1,2G} f(a_q).
\end{gather*}

The second part of the lemma follows by the associativity axiom $5.$

The claims about canonical expressions are clear.
\end{proof}

\begin{remark}
Obviously, the proof of Lemma \ref{LEMPicturesOK} also proves:
Let $\oP$ be a modular operad, $p\in\oP(C\sqcup\{e_1,\ldots,e_{2k}\},G)$ and let $\sigma$ be a permutation of pairs $\{e_1,e_2\},\ldots,\{e_{2k-1},e_{2k}\}$.
Then $\oxi{e_1e_2}\cdots\oxi{e_{2k-1}e_{2k}}p=\oxi{\sigma(e_1e_2)}\cdots\oxi{\sigma(e_{2k-1}e_{2k})}p$.
\end{remark}

We introduce one more pictorial convention:
Suppose there is a sequence $l_1,\ldots,l_n$ of (counter-clockwise) consecutive labeled points on the circle in picture such as \eqref{PICEx} such that none of them is an endpoint of an arc.
Then we can replace this sequence by a single ``bold'' point labeled by $(l_1,\ldots,l_n)$.
For example,
$$\PICAAC=\PICAAD$$
where $A:=(l_1,l_2)$ and $B:=(l_3,\ldots,l_6)$.

As an example, we rewrite the expression \eqref{EQCanFormg} in picture:
$$\PICABS$$

\begin{remark}
The above pictorial notation seemingly has another interpretation:
One might try to construct the modular envelope $\oMod{\oAss}$ of $\oAss$ directly from $\oAss$ by putting in formal results of the $\oxi{}$ operations, subject only to the relations implied by the axioms of modular operad.
The element of $\oAss$ is then depicted by the circle and the arcs depict the $\oxi{}$ operations.
The linear span of all the circles with arcs is then equipped with ``obvious'' $\ooo{}{}$ and $\oxi{}$ operations.
However, these do not satisfy the associativity axiom $6.$, thus this naive candidate for $\oMod{\oAss}$ is not even a modular operad.
Thus this interpretation of the pictures is incorrect, but it still suggests that the axiom $6.$ plays an important role.
Indeed, this is reflected in the following lemma:
\end{remark}

\begin{lemma}[The Main Lemma]
$$\PICAAE{A}{B}=\PICAAE{B}{A}$$
Informally: if an arc cuts the circle into two semicircles, the labels on both semicirles can be (independently) cyclically permuted.
\end{lemma}

\begin{proof}
This amounts to proving $\oxi{xy}f\cyc{ABxCy} = \oxi{xy}f\cyc{BAxCy}$.
We have $\cyc{ABxCy} = \cyc{BxCu} \ooo{u}{v} \cyc{vyA}$.
Then 
$$\oxi{xy}f(\cyc{BxCu} \ooo{u}{v} \cyc{vyA}) = \oxi{xy}\ooo{u}{v}(f\ot f)(\cyc{BxCu}\ot\cyc{vyA}) = \cdots$$
since $f$ is a morphism of cyclic operads,
$$\cdots=\oxi{uv}\ooo{x}{y}(f\ot f)(\cyc{BxCu}\ot\cyc{vyA})=\cdots,$$
by the associativity axiom $6.$,
$$\cdots=\oxi{uv}f(\cyc{BxCu}\ooo{x}{y}\cyc{vyA})=\oxi{uv}f\cyc{CuBAv}=\oxi{xy}f\cyc{BAxCy}.$$
The last equality is proved in the same way as Lemma \ref{LEMPicturesOK}.
\end{proof}

In fact, the proof of Theorem \ref{THMTheResult} follows in a straightforward way from The Main Lemma.
Still, we give full details below.

Here is an example of calculations using The Main Lemma:

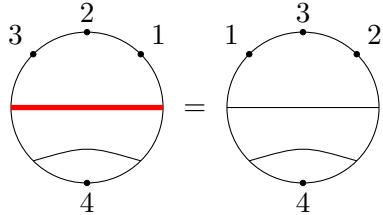
\begin{example}
$$\PICAAM=\PICAAN$$
Denote the left endpoint of the red arc by $a$, the right by $b$ and similarly $c,d$ for endpoints of the black arc.
The LHS is $\oxi{cd}\oxi{ab}f(a_q)$ with $a_q=\cyc{b123ac4d}$.
By Lemma \ref{LEMPicturesOK}, we can always have the endpoints of the red arc written as subscripts at the rightmost $\oxi{}$.
Now we apply The Main Lemma to get $\oxi{ab}f\cyc{b123ac4d}=\oxi{ab}f\cyc{b231ac4d}$.
The permutation is clear from the result.

In the sequel, we just colour one arc red on the LHS to signify an application of The Main Lemma and we leave it to the reader to figure out the permutation from the result on the RHS.
To make it easier, we always fix the red arc while passing from the LHS to the RHS (but the arc turns black on the RHS).
Finally, let's emphasize that endpoints of arcs are connected to labels rather than actual points on the circle.
Consequently, if the labels move, so do the endpoints of the arc.
For example:
$$\PICAAO=\PICAAP=\PICAAQ$$

We can also use The Main Lemma on both semicircles simultaneously:
$$\PICAAM=\PICAAR$$
\end{example}

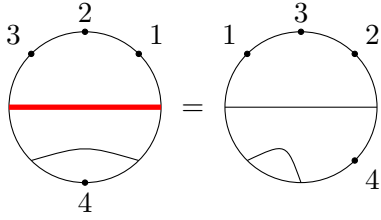
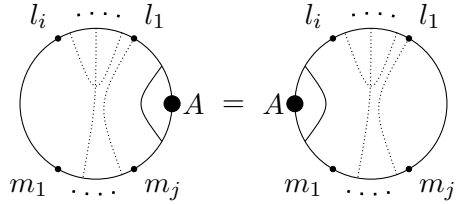
\begin{lemma} \label{LEMboundary}
$$\PICAAF=\PICAAG$$
Here the labels $l_1,\ldots,l_i,m_1,\ldots,m_j$ can be connected by any number of arcs; we don't specify the connection and this is emphasized by the dotted lines, but the connections on the LHS and RHS coincide.
Informally: call ``boundary'' a sequence of consecutive labels on the circle, the first and last of which are connected by an arc.
Then a boundary can be moved anywhere on the circle.
\end{lemma}

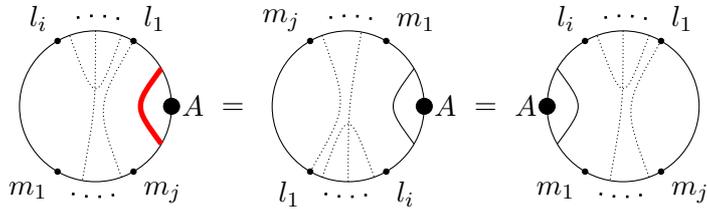
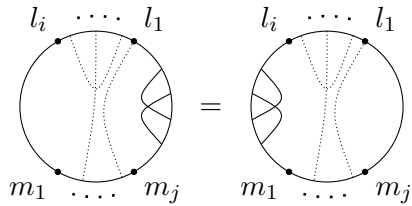
\begin{proof}
$$\PICAAL=\PICAAS=\PICAAG$$
\end{proof}

\begin{lemma} \label{LEMhandle}
$$\PICAAH=\PICAAI$$
Here the labels $l_1,\ldots,l_i,m_1,\ldots,m_j$ can be joined by any number of arcs.
Informally: call ``handle'' a sequence of four consecutive labels $a,b,c,d$ on the circle, such that $a$ is connected with $c$ by an arc and similarly $b$ with $d$.
Then a handle can be moved anywhere on the circle.
The handle will be denoted by a blue dot:
$$\PICAAJ:=\PICAAK$$
\end{lemma}

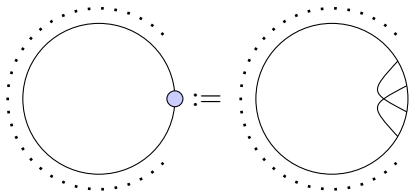
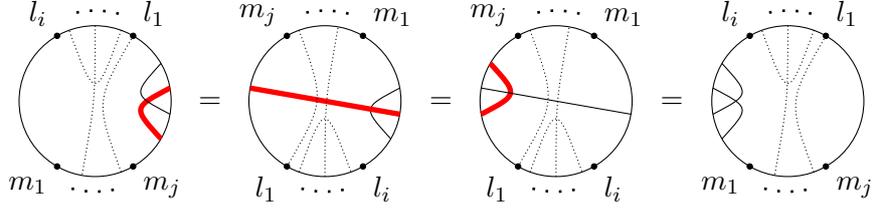
\begin{proof}
$$\PICAAT=\PICAAU=\PICAAV=\PICAAW$$
\end{proof}

Now we proceed to prove the existence part of Theorem \ref{THMTheResult}.
Given $f$, we define $\tilde{f}$ by the formula \eqref{EQCanFormg}.
We have to verify that this is a well defined morphism of modular operads.

\begin{lemma} \label{LEMgWellDefd}
The definition of $\tilde{f}$ is independent of the choice of the canonical form of $\tilde{f}(q)$.
\end{lemma}

\begin{proof}
Consider a canonical expression \eqref{EQCanFormg} with 
$$a_q:=\cyc{C_1 e_1 C_2 e_2 e_3 C_3 e_4 \cdots C_i \cdots e_{2b-3} C_b e_{2b-2} e_{2b-1} e_{2b} \cdots e_{4g+2b-2}}$$
as in \eqref{EQaq} and another canonical expression with
$$a'_q:=\cyc{C_i e_1 C_2 e_2 e_3 C_3 e_4 \cdots C_1 \cdots e_{2b-3} C_b e_{2b-2} e_{2b-1} e_{2b} \cdots e_{4g+2b-2}},$$
otherwise identical.
We verify
$$\oxi{12}\cdots\oxi{2b+4g-4,2b+4g-2}f(a_q)=\oxi{12}\cdots\oxi{2b+4g-4,2b+4g-2}f(a'_q):$$
By repeated applications of Lemma \ref{LEMboundary} and \ref{LEMhandle}, we have
$$\PICAAX \eq{\ref{LEMboundary}} \PICAAY \eq{\ref{LEMhandle}} \PICAAZ \eq{\ref{LEMboundary}} \PICABA$$

Next, consider a canonical expression \eqref{EQCanFormg} with $C_1=(c_1^1,\ldots,c_1^{k})$ as in \eqref{EQaq} and another canonical expression with $C'_1:=(c_1^2,\ldots,c_1^{k},c_1^1)$ in place of $C_1$, otherwise identical.
To simplify the notation, denote $c:=c_1^1$ and $C:=(c_1^2,\ldots,c_1^{k})$.
Again, we verify
$$\oxi{12}\cdots\oxi{2b+4g-4,2b+4g-2}f(a_q)=\oxi{12}\cdots\oxi{2b+4g-4,2b+4g-2}f(a'_q):
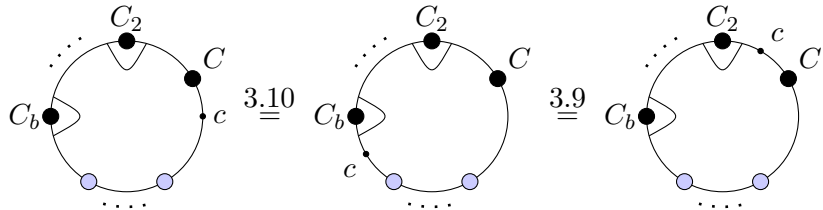
$$
$$\PICABB \eq{\ref{LEMhandle}} \PICABC \eq{\ref{LEMboundary}} \PICABD$$

These two observations easily imply independence of $\tilde{f}$ on the canonical expression.
\end{proof}

\begin{lemma} \label{LEMgMorOfModOp}
$\tilde{f}$ is a morphism of modular operads.
\end{lemma}

\begin{proof}
First, we verify that $\tilde{f}$ commutes with the action by bijections:
Let $\sigma:C\to D$ be a bijection.
It extends by identity to a bijection $\tilde{\sigma}:C\sqcup\{e_1,\ldots,e_{2G}\}\to D\sqcup\{e_1,\ldots,e_{2G}\}$.
Let $q\in\oQO(C,G)$ have a canonical expression as in \eqref{EQCanForm}:
$$q = \oxi{12}\oxi{34}\cdots(a_q).$$
Then, by the equivariance of $\oxi{}$,
$$\sigma q = \sigma\oxi{12}\oxi{34}\cdots(a_q) = \oxi{12}\oxi{34}\cdots(\tilde{\sigma}a_q).$$
Hence $\sigma q = \oxi{12}\oxi{34}\cdots(\tilde{\sigma}a_q)$ is a canonical expression of $\sigma q\in\oQO(D,G)$.
This justifies the last equality in
\begin{gather*}
\sigma \tilde{f}(q) = \sigma\oxi{12}\oxi{34}\cdots f(a_q) = \oxi{12}\oxi{34}\cdots \tilde{\sigma}f(a_q) = \oxi{12}\oxi{34}\cdots f(\tilde{\sigma} a_q) = \tilde{f}(\sigma q).
\end{gather*}

Second, we verify that $\tilde{f}$ commutes with $\ooo{}{}$:
Let $q=\{\cyc{c,C},\cyc{C_2},\ldots,\cyc{C_b}\}^g$ and $q'=\{\cyc{C',c'},\cyc{C'_2},\ldots,\cyc{C'_{b'}}\}^{g'}$.
Then
\begin{gather*}
\tilde{f}(q) \ooo{c}{c'} \tilde{f}(q') = \PICABE \ooo{c}{c'} \PICABF = \\
= \PICABG = \PICABH = \\
\eq{\substack{\ref{LEMboundary}\\ \ref{LEMhandle}}} \PICABI = \\
= \tilde{f}(\{\cyc{C',C},\cyc{C_2},\ldots,\cyc{C_b},\cyc{C'_2},\ldots,\cyc{C'_{b'}}\}^{g+g'}) = \\
= \tilde{f}(q\ooo{c}{c'}q')
\end{gather*}
To justify the equality of the first and second line, we rewrite the first line (using $a_q$ in the canonical expression of $q$ and $a'_{q'}$ in the canonical expression of $q'$) using the associativity axiom $7.$ (and its mirror image $7'.$) repeatedly in the following way:
\begin{gather*}
\oxi{12}\oxi{34}\cdots f(a_q) \ooo{c}{c'} \oxi{1'2'}\oxi{3'4'}\cdots f(a'_{q'}) = \ooo{c}{c'} \left( \oxi{12}\oxi{34}\cdots f(a_q) \ot \oxi{1'2'}\oxi{3'4'}\cdots f(a'_{q'}) \right) = \\
\eq{7.} \oxi{12} \ooo{c}{c'} \left( \oxi{34}\cdots f(a_q) \ot \oxi{1'2'}\oxi{3'4'}\cdots f(a'_{q'}) \right) \eq{7'.} \oxi{12}\oxi{1'2'} \ooo{c}{c'} \left( \oxi{34}\cdots f(a_q) \ot \oxi{3'4'}\cdots f(a'_{q'}) \right) = \\
\cdots = \oxi{12}\oxi{1'2'}\oxi{34}\oxi{3'4'}\cdots\ooo{c}{c'} \left( f(a_q) \ot f(a'_{q'}) \right) \eq{5.} \oxi{12}\oxi{34}\cdots\oxi{1'2'}\oxi{3'4'}\cdots \ooo{c}{c'} \left( f(a_q)\ot f(a'_{q'}) \right) = \\
= \oxi{12}\oxi{34}\cdots\oxi{1'2'}\oxi{3'4'}\cdots f(a_q \ooo{c}{c'} a'_{q'})
\end{gather*}
and this corresponds to the picture on the RHS of the second line.
The LHS picture with merging circles just reminds us how $\ooo{c}{c'}$ in $\oAss$ is defined.

Last, we verify that $\tilde{f}$ commutes with $\oxi{}$.
There are two cases:

If $q=\{\cyc{aAbB},\cyc{C_2},\ldots,\cyc{C_b}\}^g$, then $\oxi{ab}q=\{\cyc{B},\cyc{A},\cyc{C_2},\ldots,\cyc{C_b}\}^g$ and hence
\begin{gather*}
\oxi{ab}\tilde{f}(q) = \oxi{ab} \PICABJ = \PICABK = \PICABL = \tilde{f}(\oxi{ab}q)
\end{gather*}

If $q=\{\cyc{aA},\cyc{bB},\cyc{C_3},\ldots,\cyc{C_b}\}^g$, then $\oxi{ab}q=\{\cyc{BA},\cyc{C_3},\ldots,\cyc{C_b}\}^{g+1}$ and hence
\begin{gather*}
\oxi{ab}\tilde{f}(q) = \oxi{ab}\PICABM = \PICABN = \PICABO = \\
= \PICABP = \PICABQ = \PICABR = \tilde{f}(\oxi{ab}q)
\end{gather*}

In the case $\cM$ is the category of chain complexes, $\oQO$ has zero differential and
\begin{gather*}
\partial\tilde{f}(q)=\partial\oxi{12}\oxi{34}\cdots f(a_q)=\oxi{12}\oxi{34}\cdots\partial f(a_q)=\oxi{12}\oxi{34}\cdots f(\partial a_q)=0=\tilde{f}(\partial q),
\end{gather*}
thus $\tilde{f}$ is a morphism of chain complexes.
\end{proof}

Finally, Lemmas \ref{LEMCanFormg},\ref{LEMgWellDefd} and \ref{LEMgMorOfModOp} constitute a proof of Theorem \ref{THMTheResult}.


\end{document}

%% file: commands.tex
\newcommand{\ot}{\otimes}

\newcommand{\otexp}[2]{{#1^{\ot #2}}}

\makeatletter
\newcommand{\catVec}{\@ifnextchar[{\bracketCVec}{\nobracketCVec}}
\def\bracketCVec[#1]{ \mathsf{Vec}_{#1} }
\def\nobracketCVec{ \mathsf{Vec} }
\makeatother

\makeatletter
\newcommand{\downar}{\@ifnextchar*{\starDownAr}{\nostarDownAr}}
\def\starDownAr*{ {\raisebox{0.5pt}[0pt][0pt]{\ensuremath{\scriptstyle{\downarrow}}}} }
\def\nostarDownAr{ {\raisebox{1pt}[0pt][0pt]{\ensuremath{\downarrow}}} }
\makeatother

\makeatletter
\newcommand{\upar}{\@ifnextchar*{\starUpAr}{\nostarUpAr}}
\def\starUpAr*{ {\raisebox{1pt}[0pt][0pt]{\ensuremath{\scriptstyle{\uparrow}}}} }
\def\nostarUpAr{ {\raisebox{1pt}[0pt][0pt]{\ensuremath{\uparrow}}} }
\makeatother

\makeatletter
\newcommand{\id}{\@ifnextchar[{\tensorID}{\notensorID}}
\def\tensorID[#1]{ \otexp{\mathsf{1}}{#1} }
\def\notensorID{ \mathsf{1}	}
\makeatother

\newcommand{\set}[2]{\left\{ #1 \ | \ #2 \right\} }

\newcommand{\N}{\mathbb{N}}

\makeatletter
\newcommand{\quism}{\@ifnextchar[{\quismName}{\quismNoName}}
\def\quismName[#1]{\xrightarrow[#1]{\sim}}
\def\quismNoName{\xrightarrow{\sim}}
\makeatother

\newcommand{\eq}[1]{\stackrel{#1}{=}}

\newcommand{\oAss}{\mathcal{A}\mathit{ss}}

\newcommand{\oCom}{\mathcal{C}\mathit{om}}

\makeatletter
\newcommand{\smod}
	{\@ifnextchar*{\smod@nonsigma}{\smod@sigma}}
\newcommand{\smod@nonsigma}[1]	
	{\@ifnextchar[{\smod@nonsigma@col}{\smod@nonsigma@uncol}}
\newcommand{\smod@sigma}
	{\@ifnextchar[{\smod@sigma@col}{\smod@sigma@uncol}}
\def\smod@sigma@col[#1]{{\mbox{\ensuremath{#1}-\ensuremath{\Sigma}-module}}}
\def\smod@nonsigma@col[#1]{{\mbox{\ensuremath{#1}-collection}}}
\def\smod@sigma@uncol{\mbox{\ensuremath{\Sigma}-module} }
\def\smod@nonsigma@uncol{{\mbox{collection}} }
\makeatother

\makeatletter
\newcommand{\Fr}{\@ifnextchar[{\Fr@weight}{\Fr@noweight}}
\def\Fr@weight[#1]#2{\mathbb{F}^{#1}(#2)}
\def\Fr@noweight#1{\mathbb{F}(#1)}
\makeatother

\newcommand{\oo}{\circ}

\newcommand{\ooo}[2]{\sideset{_{#1}}{_{#2}}{\mathop{\oo}}}

\newcommand{\oxi}[1]{\xi_{#1}}

\makeatletter
\newcommand{\TJ}{\@ifnextchar*{\TJ@star}{\TJ@nostar}}
\def\TJ@nostar#1{\textrm{TJ}(#1)}
\def\TJ@star*{\textrm{TJ}}
\makeatother



%% file: ModAss_pic.tex
\newcommand{\PICAAA}
{
\begin{tikzpicture}[baseline=-\the\dimexpr\fontdimen22\textfont2\relax]
\draw (0,0) circle (1cm);
\filldraw (0:1) circle (1pt) node[right]{$1$};
\filldraw (45:1) circle (1pt) node[above right ]{$2$};
\filldraw (90:1) circle (1pt) node[above]{$3$};
\filldraw (135:1) circle (1pt) node[above left]{$4$};
\filldraw (180:1) circle (1pt) node[left]{$5$};
\filldraw (225:1) circle (1pt) node[below left]{$6$};
\filldraw (270:1) circle (1pt) node[below]{$7$};
\filldraw (315:1) circle (1pt) node[below right]{$8$};
\draw (45:1) .. controls (0,0) .. (180:1);
\draw (90:1) .. controls (0,0) .. (315:1);
\draw (225:1) .. controls (247.5:.5) .. (270:1);
\end{tikzpicture}
}

\newcommand{\PICAAB}
{
\begin{tikzpicture}[baseline=-\the\dimexpr\fontdimen22\textfont2\relax]
\draw (0,0) circle (1cm);
\filldraw (0:1) circle (1pt) node[right]{$1$};
\filldraw (135:1) circle (1pt) node[above left]{$4$};
\draw (45:1) .. controls (0,0) .. (180:1);
\draw (90:1) .. controls (0,0) .. (315:1);
\draw (225:1) .. controls (247.5:.5) .. (270:1);
\end{tikzpicture}
}

\newcommand{\PICAAC}
{
\begin{tikzpicture}[baseline=-\the\dimexpr\fontdimen22\textfont2\relax]
\draw (0,0) circle (1cm);
\filldraw (0:1) circle (1pt) node[right]{$l_6$};
\filldraw (90:1) circle (1pt) node[above]{$l_1$};
\filldraw (135:1) circle (1pt) node[above left]{$l_2$};
\filldraw (225:1) circle (1pt) node[below left]{$l_3$};
\filldraw (270:1) circle (1pt) node[below]{$l_4$};
\filldraw (315:1) circle (1pt) node[below right]{$l_5$};
\draw (45:1) .. controls (0,0) .. (180:1);
\end{tikzpicture}
}

\newcommand{\PICAAD}
{
\begin{tikzpicture}[baseline=-\the\dimexpr\fontdimen22\textfont2\relax]
\draw (0,0) circle (1cm);
\filldraw (112.5:1) circle (3pt) node[above left]{$A$};
\filldraw (292.5:1) circle (3pt) node[below right]{$B$};
\draw (45:1) .. controls (0,0) .. (180:1);
\end{tikzpicture}
}

\newcommand{\PICAAE}[2]
{
\begin{tikzpicture}[baseline=-\the\dimexpr\fontdimen22\textfont2\relax]
\draw (0,0) circle (1cm);
\filldraw (330:1) circle (3pt) node[below right]{$#1$};
\filldraw (30:1) circle (3pt) node[above right]{$#2$};
\filldraw (180:1) circle (3pt) node[left]{$C$};
\draw (90:1) .. controls (0,0) .. (270:1);
\end{tikzpicture}
}

\newcommand{\PICAAF}
{
\begin{tikzpicture}[baseline=-\the\dimexpr\fontdimen22\textfont2\relax]
\draw (0,0) circle (1cm);
\filldraw (0:1) circle (3pt) node[right]{$A$};
\draw (-30:1) .. controls (0:.5) .. (30:1);
\filldraw (60:1) circle (1pt) node[above right]{$l_1$};
\draw[line width=1pt, loosely dotted] (75:1.2) arc[radius=1.2cm,start angle=75, end angle=105];
\filldraw (120:1) circle (1pt) node[above left]{$l_i$};
\filldraw (240:1) circle (1pt) node[below left]{$m_1$};
\draw[line width=1pt, loosely dotted] (255:1.2) arc[radius=1.2cm,start angle=255, end angle=285];
\filldraw (300:1) circle (1pt) node[below right]{$m_j$};
\draw[densely dotted] (60:1) .. controls (0,0) .. (290:1);
\draw[densely dotted] (70:1) .. controls (0,0) .. (110:1);
\draw[densely dotted] (90:1) .. controls (0,0) .. (260:1);
\end{tikzpicture}
}

\newcommand{\PICAAG}
{
\begin{tikzpicture}[baseline=-\the\dimexpr\fontdimen22\textfont2\relax]
\draw (0,0) circle (1cm);
\filldraw (180:1) circle (3pt) node[left]{$A$};
\draw (150:1) .. controls (180:.5) .. (210:1);
\filldraw (60:1) circle (1pt) node[above right]{$l_1$};
\draw[line width=1pt, loosely dotted] (75:1.2) arc[radius=1.2cm,start angle=75, end angle=105];
\filldraw (120:1) circle (1pt) node[above left]{$l_i$};
\filldraw (240:1) circle (1pt) node[below left]{$m_1$};
\draw[line width=1pt, loosely dotted] (255:1.2) arc[radius=1.2cm,start angle=255, end angle=285];
\filldraw (300:1) circle (1pt) node[below right]{$m_j$};
\draw[densely dotted] (60:1) .. controls (0,0) .. (290:1);
\draw[densely dotted] (70:1) .. controls (0,0) .. (110:1);
\draw[densely dotted] (90:1) .. controls (0,0) .. (260:1);
\end{tikzpicture}
}

\newcommand{\PICAAH}
{
\begin{tikzpicture}[baseline=-\the\dimexpr\fontdimen22\textfont2\relax]
\draw (0,0) circle (1cm);
\draw (-30:1) .. controls (-10:.5) .. (10:1);
\draw (-10:1) .. controls (10:.5) .. (30:1);
\filldraw (60:1) circle (1pt) node[above right]{$l_1$};
\draw[line width=1pt, loosely dotted] (75:1.2) arc[radius=1.2cm,start angle=75, end angle=105];
\filldraw (120:1) circle (1pt) node[above left]{$l_i$};
\filldraw (240:1) circle (1pt) node[below left]{$m_1$};
\draw[line width=1pt, loosely dotted] (255:1.2) arc[radius=1.2cm,start angle=255, end angle=285];
\filldraw (300:1) circle (1pt) node[below right]{$m_j$};
\draw[densely dotted] (60:1) .. controls (0,0) .. (290:1);
\draw[densely dotted] (70:1) .. controls (0,0) .. (110:1);
\draw[densely dotted] (90:1) .. controls (0,0) .. (260:1);
\end{tikzpicture}
}

\newcommand{\PICAAI}
{
\begin{tikzpicture}[baseline=-\the\dimexpr\fontdimen22\textfont2\relax]
\draw (0,0) circle (1cm);
\draw (150:1) .. controls (170:.5) .. (190:1);
\draw (170:1) .. controls (190:.5) .. (210:1);
\filldraw (60:1) circle (1pt) node[above right]{$l_1$};
\draw[line width=1pt, loosely dotted] (75:1.2) arc[radius=1.2cm,start angle=75, end angle=105];
\filldraw (120:1) circle (1pt) node[above left]{$l_i$};
\filldraw (240:1) circle (1pt) node[below left]{$m_1$};
\draw[line width=1pt, loosely dotted] (255:1.2) arc[radius=1.2cm,start angle=255, end angle=285];
\filldraw (300:1) circle (1pt) node[below right]{$m_j$};
\draw[densely dotted] (60:1) .. controls (0,0) .. (290:1);
\draw[densely dotted] (70:1) .. controls (0,0) .. (110:1);
\draw[densely dotted] (90:1) .. controls (0,0) .. (260:1);
\end{tikzpicture}
}

\newcommand{\PICAAJ}
{
\begin{tikzpicture}[baseline=-\the\dimexpr\fontdimen22\textfont2\relax]
\draw (0,0) circle (1cm);
\draw[line width=1pt, loosely dotted] (45:1.2) arc[radius=1.2cm,start angle=45, end angle=315];
\filldraw[fill=blue!20] (0:1) circle (3pt);
\end{tikzpicture}
}

\newcommand{\PICAAK}
{
\begin{tikzpicture}[baseline=-\the\dimexpr\fontdimen22\textfont2\relax]
\draw (0,0) circle (1cm);
\draw[line width=1pt, loosely dotted] (45:1.2) arc[radius=1.2cm,start angle=45, end angle=315];
\draw (-30:1) .. controls (-10:.5) .. (10:1);
\draw (-10:1) .. controls (10:.5) .. (30:1);
\end{tikzpicture}
}

\newcommand{\PICAAL}
{
\begin{tikzpicture}[baseline=-\the\dimexpr\fontdimen22\textfont2\relax]
\draw (0,0) circle (1cm);
\filldraw (0:1) circle (3pt) node[right]{$A$};
\draw[line width=2pt,red] (-30:1) .. controls (0:.5) .. (30:1);
\filldraw (60:1) circle (1pt) node[above right]{$l_1$};
\draw[line width=1pt, loosely dotted] (75:1.2) arc[radius=1.2cm,start angle=75, end angle=105];
\filldraw (120:1) circle (1pt) node[above left]{$l_i$};
\filldraw (240:1) circle (1pt) node[below left]{$m_1$};
\draw[line width=1pt, loosely dotted] (255:1.2) arc[radius=1.2cm,start angle=255, end angle=285];
\filldraw (300:1) circle (1pt) node[below right]{$m_j$};
\draw[densely dotted] (60:1) .. controls (0,0) .. (290:1);
\draw[densely dotted] (70:1) .. controls (0,0) .. (110:1);
\draw[densely dotted] (90:1) .. controls (0,0) .. (260:1);
\end{tikzpicture}
}

\newcommand{\PICAAM}
{
\begin{tikzpicture}[baseline=-\the\dimexpr\fontdimen22\textfont2\relax]
\draw (0,0) circle (1cm);
\filldraw (45:1) circle (1pt) node[above right ]{$1$};
\filldraw (90:1) circle (1pt) node[above]{$2$};
\filldraw (135:1) circle (1pt) node[above left]{$3$};
\filldraw (270:1) circle (1pt) node[below]{$4$};
\draw[line width=2pt,red] (0:1) .. controls (0,0) .. (180:1);
\draw (225:1) .. controls (270:.5) .. (315:1);
\end{tikzpicture}
}

\newcommand{\PICAAN}
{
\begin{tikzpicture}[baseline=-\the\dimexpr\fontdimen22\textfont2\relax]
\draw (0,0) circle (1cm);
\filldraw (45:1) circle (1pt) node[above right ]{$2$};
\filldraw (90:1) circle (1pt) node[above]{$3$};
\filldraw (135:1) circle (1pt) node[above left]{$1$};
\filldraw (270:1) circle (1pt) node[below]{$4$};
\draw (0:1) .. controls (0,0) .. (180:1);
\draw (225:1) .. controls (270:.5) .. (315:1);
\end{tikzpicture}
}

\newcommand{\PICAAO}
{
\begin{tikzpicture}[baseline=-\the\dimexpr\fontdimen22\textfont2\relax]
\draw (0,0) circle (1cm);
\filldraw (45:1) circle (1pt) node[above right ]{$2$};
\filldraw (90:1) circle (1pt) node[above]{$3$};
\filldraw (135:1) circle (1pt) node[above left]{$1$};
\filldraw (270:1) circle (1pt) node[below]{$4$};
\draw (0:1) .. controls (0,0) .. (180:1);
\draw[line width=2pt,red] (225:1) .. controls (270:.5) .. (315:1);
\end{tikzpicture}
}

\newcommand{\PICAAP}
{
\begin{tikzpicture}[baseline=-\the\dimexpr\fontdimen22\textfont2\relax]
\draw (0,0) circle (1cm);
\filldraw (90:1) circle (1pt) node[above]{$2$};
\filldraw (135:1) circle (1pt) node[above left]{$3$};
\filldraw (180:1) circle (1pt) node[left]{$1$};
\filldraw (270:1) circle (1pt) node[below]{$4$};
\draw[line width=2pt,red] (0:1) .. controls (22.5:.5) .. (45:1);
\draw (225:1) .. controls (270:.5) .. (315:1);
\end{tikzpicture}
}

\newcommand{\PICAAQ}
{
\begin{tikzpicture}[baseline=-\the\dimexpr\fontdimen22\textfont2\relax]
\draw (0,0) circle (1cm);
\filldraw (90:1) circle (1pt) node[above]{$3$};
\filldraw (135:1) circle (1pt) node[above left]{$1$};
\filldraw (225:1) circle (1pt) node[below left]{$4$};
\filldraw (315:1) circle (1pt) node[below right]{$2$};
\draw (0:1) .. controls (22.5:.5) .. (45:1);
\draw (180:1) .. controls (225:.5) .. (270:1);
\end{tikzpicture}
}

\newcommand{\PICAAR}
{
\begin{tikzpicture}[baseline=-\the\dimexpr\fontdimen22\textfont2\relax]
\draw (0,0) circle (1cm);
\filldraw (45:1) circle (1pt) node[above right ]{$2$};
\filldraw (90:1) circle (1pt) node[above]{$3$};
\filldraw (135:1) circle (1pt) node[above left]{$1$};
\filldraw (315:1) circle (1pt) node[below right]{$4$};
\draw (0:1) .. controls (0,0) .. (180:1);
\draw (225:1) .. controls (247.5:.5) .. (270:1);
\end{tikzpicture}
}

\newcommand{\PICAAS}
{
\begin{tikzpicture}[baseline=-\the\dimexpr\fontdimen22\textfont2\relax]
\draw (0,0) circle (1cm);
\filldraw (0:1) circle (3pt) node[right]{$A$};
\draw (-30:1) .. controls (0:.5) .. (30:1);
\filldraw (240:1) circle (1pt) node[below left]{$l_1$};
\draw[line width=1pt, loosely dotted] (255:1.2) arc[radius=1.2cm,start angle=255, end angle=285];
\filldraw (300:1) circle (1pt) node[below right]{$l_i$};
\filldraw (60:1) circle (1pt) node[above right]{$m_1$};
\draw[line width=1pt, loosely dotted] (75:1.2) arc[radius=1.2cm,start angle=75, end angle=105];
\filldraw (120:1) circle (1pt) node[above left]{$m_j$};
\draw[densely dotted] (240:1) .. controls (0,0) .. (110:1);
\draw[densely dotted] (250:1) .. controls (0,0) .. (290:1);
\draw[densely dotted] (270:1) .. controls (0,0) .. (80:1);
\end{tikzpicture}
}

\newcommand{\PICAAT}
{
\begin{tikzpicture}[baseline=-\the\dimexpr\fontdimen22\textfont2\relax]
\draw (0,0) circle (1cm);
\draw[line width=2pt,red] (-30:1) .. controls (-10:.5) .. (10:1);
\draw (-10:1) .. controls (10:.5) .. (30:1);
\filldraw (60:1) circle (1pt) node[above right]{$l_1$};
\draw[line width=1pt, loosely dotted] (75:1.2) arc[radius=1.2cm,start angle=75, end angle=105];
\filldraw (120:1) circle (1pt) node[above left]{$l_i$};
\filldraw (240:1) circle (1pt) node[below left]{$m_1$};
\draw[line width=1pt, loosely dotted] (255:1.2) arc[radius=1.2cm,start angle=255, end angle=285];
\filldraw (300:1) circle (1pt) node[below right]{$m_j$};
\draw[densely dotted] (60:1) .. controls (0,0) .. (290:1);
\draw[densely dotted] (70:1) .. controls (0,0) .. (110:1);
\draw[densely dotted] (90:1) .. controls (0,0) .. (260:1);
\end{tikzpicture}
}

\newcommand{\PICAAU}
{
\begin{tikzpicture}[baseline=-\the\dimexpr\fontdimen22\textfont2\relax]
\draw (0,0) circle (1cm);
\draw (-30:1) .. controls (-10:.5) .. (10:1);
\draw[line width=2pt,red] (-10:1) .. controls (0,0) .. (170:1);
\filldraw (240:1) circle (1pt) node[below left]{$l_1$};
\draw[line width=1pt, loosely dotted] (255:1.2) arc[radius=1.2cm,start angle=255, end angle=285];
\filldraw (300:1) circle (1pt) node[below right]{$l_i$};
\filldraw (60:1) circle (1pt) node[above right]{$m_1$};
\draw[line width=1pt, loosely dotted] (75:1.2) arc[radius=1.2cm,start angle=75, end angle=105];
\filldraw (120:1) circle (1pt) node[above left]{$m_j$};
\draw[densely dotted] (240:1) .. controls (0,0) .. (110:1);
\draw[densely dotted] (250:1) .. controls (0,0) .. (290:1);
\draw[densely dotted] (270:1) .. controls (0,0) .. (80:1);
\end{tikzpicture}
}

\newcommand{\PICAAV}
{
\begin{tikzpicture}[baseline=-\the\dimexpr\fontdimen22\textfont2\relax]
\draw (0,0) circle (1cm);
\draw[line width=2pt,red] (150:1) .. controls (170:.5) .. (190:1);
\draw (-10:1) .. controls (0,0) .. (170:1);
\filldraw (240:1) circle (1pt) node[below left]{$l_1$};
\draw[line width=1pt, loosely dotted] (255:1.2) arc[radius=1.2cm,start angle=255, end angle=285];
\filldraw (300:1) circle (1pt) node[below right]{$l_i$};
\filldraw (60:1) circle (1pt) node[above right]{$m_1$};
\draw[line width=1pt, loosely dotted] (75:1.2) arc[radius=1.2cm,start angle=75, end angle=105];
\filldraw (120:1) circle (1pt) node[above left]{$m_j$};
\draw[densely dotted] (240:1) .. controls (0,0) .. (110:1);
\draw[densely dotted] (250:1) .. controls (0,0) .. (290:1);
\draw[densely dotted] (270:1) .. controls (0,0) .. (80:1);
\end{tikzpicture}
}

\newcommand{\PICAAW}
{
\begin{tikzpicture}[baseline=-\the\dimexpr\fontdimen22\textfont2\relax]
\draw (0,0) circle (1cm);
\draw (150:1) .. controls (170:.5) .. (190:1);
\draw (170:1) .. controls (190:.5) .. (210:1);
\filldraw (60:1) circle (1pt) node[above right]{$l_1$};
\draw[line width=1pt, loosely dotted] (75:1.2) arc[radius=1.2cm,start angle=75, end angle=105];
\filldraw (120:1) circle (1pt) node[above left]{$l_i$};
\filldraw (240:1) circle (1pt) node[below left]{$m_1$};
\draw[line width=1pt, loosely dotted] (255:1.2) arc[radius=1.2cm,start angle=255, end angle=285];
\filldraw (300:1) circle (1pt) node[below right]{$m_j$};
\draw[densely dotted] (60:1) .. controls (0,0) .. (290:1);
\draw[densely dotted] (70:1) .. controls (0,0) .. (110:1);
\draw[densely dotted] (90:1) .. controls (0,0) .. (260:1);
\end{tikzpicture}
}

\newcommand{\PICAAX}
{
\begin{tikzpicture}[baseline=-\the\dimexpr\fontdimen22\textfont2\relax]
\draw (0,0) circle (1cm);
\filldraw (0:1) circle (3pt) node[right]{$C_1$};
\filldraw (45:1) circle (3pt) node[above right]{$C_2$};
\draw[line width=1pt, loosely dotted] (63.75:1.2) arc[radius=1.2cm,start angle=63.75, end angle=93.75];
\filldraw (112.5:1) circle (3pt) node[above left]{$C_i$};
\draw[line width=1pt, loosely dotted] (131.25:1.2) arc[radius=1.2cm,start angle=131.25, end angle=161.25];
\filldraw (180:1) circle (3pt) node[left]{$C_b$};
\filldraw[fill=blue!20] (225:1) circle (3pt);
\draw[line width=1pt, loosely dotted] (255:1.2) arc[radius=1.2cm,start angle=255, end angle=285];
\filldraw[fill=blue!20] (315:1) circle (3pt);
\draw (30:1) .. controls (45:.5) .. (60:1);
\draw (97:1) .. controls (112.5:.5) .. (127:1);
\draw (165:1) .. controls (180:.5) .. (195:1);
\end{tikzpicture}
}

\newcommand{\PICAAY}
{
\begin{tikzpicture}[baseline=-\the\dimexpr\fontdimen22\textfont2\relax]
\draw (0,0) circle (1cm);
\filldraw (0:1) circle (3pt) node[right]{$C_1$};
\filldraw (45:1) circle (3pt) node[above right]{$C_i$};
\filldraw (90.5:1) circle (3pt) node[above]{$C_2$};
\draw[line width=1pt, loosely dotted] (120:1.2) arc[radius=1.2cm,start angle=120, end angle=150];
\draw (135:1.1) node[above left]{$\widehat{C_i}$};
\filldraw (180:1) circle (3pt) node[left]{$C_b$};
\filldraw[fill=blue!20] (225:1) circle (3pt);
\draw[line width=1pt, loosely dotted] (255:1.2) arc[radius=1.2cm,start angle=255, end angle=285];
\filldraw[fill=blue!20] (315:1) circle (3pt);
\draw (30:1) .. controls (0,0) .. (200:1);
\draw (75:1) .. controls (90:.5) .. (105:1);
\draw (165:1) .. controls (180:.5) .. (195:1);
\end{tikzpicture}
}

\newcommand{\PICAAZ}
{
\begin{tikzpicture}[baseline=-\the\dimexpr\fontdimen22\textfont2\relax]
\draw (0,0) circle (1cm);
\filldraw (0:1) circle (3pt) node[right]{$C_1$};
\filldraw (45:1) circle (3pt) node[above right]{$C_i$};
\filldraw (90.5:1) circle (3pt) node[above]{$C_2$};
\draw[line width=1pt, loosely dotted] (120:1.2) arc[radius=1.2cm,start angle=120, end angle=150];
\draw (135:1.1) node[above left]{$\widehat{C_i}$};
\filldraw (180:1) circle (3pt) node[left]{$C_b$};
\filldraw[fill=blue!20] (225:1) circle (3pt);
\draw[line width=1pt, loosely dotted] (255:1.2) arc[radius=1.2cm,start angle=255, end angle=285];
\filldraw[fill=blue!20] (315:1) circle (3pt);
\draw (-15:1) .. controls (0:.5) .. (15:1);
\draw (75:1) .. controls (90:.5) .. (105:1);
\draw (165:1) .. controls (180:.5) .. (195:1);
\end{tikzpicture}
}

\newcommand{\PICABA}
{
\begin{tikzpicture}[baseline=-\the\dimexpr\fontdimen22\textfont2\relax]
\draw (0,0) circle (1cm);
\filldraw (0:1) circle (3pt) node[right]{$C_i$};
\filldraw (45:1) circle (3pt) node[above right]{$C_2$};
\draw[line width=1pt, loosely dotted] (63.75:1.2) arc[radius=1.2cm,start angle=63.75, end angle=93.75];
\filldraw (112.5:1) circle (3pt) node[above left]{$C_1$};
\draw[line width=1pt, loosely dotted] (131.25:1.2) arc[radius=1.2cm,start angle=131.25, end angle=161.25];
\filldraw (180:1) circle (3pt) node[left]{$C_b$};
\filldraw[fill=blue!20] (225:1) circle (3pt);
\draw[line width=1pt, loosely dotted] (255:1.2) arc[radius=1.2cm,start angle=255, end angle=285];
\filldraw[fill=blue!20] (315:1) circle (3pt);
\draw (30:1) .. controls (45:.5) .. (60:1);
\draw (97:1) .. controls (112.5:.5) .. (127:1);
\draw (165:1) .. controls (180:.5) .. (195:1);
\end{tikzpicture}
}

\newcommand{\PICABB}
{
\begin{tikzpicture}[baseline=-\the\dimexpr\fontdimen22\textfont2\relax]
\draw (0,0) circle (1cm);
\filldraw (0:1) circle (1pt) node[right]{$c$};
\filldraw (30:1) circle (3pt) node[above right]{$C$};
\filldraw (90:1) circle (3pt) node[above]{$C_2$};
\draw[line width=1pt, loosely dotted] (120:1.2) arc[radius=1.2cm,start angle=120, end angle=150];
\filldraw (180:1) circle (3pt) node[left]{$C_b$};
\filldraw[fill=blue!20] (240:1) circle (3pt);
\draw[line width=1pt, loosely dotted] (255:1.2) arc[radius=1.2cm,start angle=255, end angle=285];
\filldraw[fill=blue!20] (300:1) circle (3pt);
\draw (75:1) .. controls (90:.5) .. (105:1);
\draw (165:1) .. controls (180:.5) .. (195:1);
\end{tikzpicture}
}

\newcommand{\PICABC}
{
\begin{tikzpicture}[baseline=-\the\dimexpr\fontdimen22\textfont2\relax]
\draw (0,0) circle (1cm);
\filldraw (210:1) circle (1pt) node[below left]{$c$};
\filldraw (30:1) circle (3pt) node[above right]{$C$};
\filldraw (90:1) circle (3pt) node[above]{$C_2$};
\draw[line width=1pt, loosely dotted] (120:1.2) arc[radius=1.2cm,start angle=120, end angle=150];
\filldraw (180:1) circle (3pt) node[left]{$C_b$};
\filldraw[fill=blue!20] (240:1) circle (3pt);
\draw[line width=1pt, loosely dotted] (255:1.2) arc[radius=1.2cm,start angle=255, end angle=285];
\filldraw[fill=blue!20] (300:1) circle (3pt);
\draw (75:1) .. controls (90:.5) .. (105:1);
\draw (165:1) .. controls (180:.5) .. (195:1);
\end{tikzpicture}
}

\newcommand{\PICABD}
{
\begin{tikzpicture}[baseline=-\the\dimexpr\fontdimen22\textfont2\relax]
\draw (0,0) circle (1cm);
\filldraw (60:1) circle (1pt) node[above right]{$c$};
\filldraw (30:1) circle (3pt) node[above right]{$C$};
\filldraw (90:1) circle (3pt) node[above]{$C_2$};
\draw[line width=1pt, loosely dotted] (120:1.2) arc[radius=1.2cm,start angle=120, end angle=150];
\filldraw (180:1) circle (3pt) node[left]{$C_b$};
\filldraw[fill=blue!20] (240:1) circle (3pt);
\draw[line width=1pt, loosely dotted] (255:1.2) arc[radius=1.2cm,start angle=255, end angle=285];
\filldraw[fill=blue!20] (300:1) circle (3pt);
\draw (75:1) .. controls (90:.5) .. (105:1);
\draw (165:1) .. controls (180:.5) .. (195:1);
\end{tikzpicture}
}

\newcommand{\PICABE}
{
\begin{tikzpicture}[baseline=-\the\dimexpr\fontdimen22\textfont2\relax]
\draw (0,0) circle (1cm);
\filldraw (0:1) circle (1pt) node[right]{$c$};
\filldraw (45:1) circle (3pt) node[above right]{$C$};
\filldraw (120:1) circle (3pt) node[above]{$C_2$};
\draw[line width=1pt, loosely dotted] (150:1.2) arc[radius=1.2cm,start angle=150, end angle=180];
\filldraw (210:1) circle (3pt) node[left]{$C_b$};
\filldraw[fill=blue!20] (270:1) circle (3pt);
\draw[line width=1pt, loosely dotted] (285:1.2) arc[radius=1.2cm,start angle=285, end angle=315];
\draw (300:1.5) node{$g$};
\filldraw[fill=blue!20] (330:1) circle (3pt);
\draw (105:1) .. controls (120:.5) .. (135:1);
\draw (195:1) .. controls (210:.5) .. (225:1);
\end{tikzpicture}
}

\newcommand{\PICABF}
{
\begin{tikzpicture}[baseline=-\the\dimexpr\fontdimen22\textfont2\relax]
\draw (0,0) circle (1cm);
\filldraw (135:1) circle (3pt) node[above left]{$C'$};
\filldraw (180:1) circle (1pt) node[left]{$c'$};
\filldraw (240:1) circle (3pt) node[below]{$C'_2$};
\draw[line width=1pt, loosely dotted] (270:1.2) arc[radius=1.2cm,start angle=270, end angle=300];
\filldraw (-30:1) circle (3pt) node[right]{$C'_{b'}$};
\filldraw[fill=blue!20] (30:1) circle (3pt);
\draw[line width=1pt, loosely dotted] (45:1.2) arc[radius=1.2cm,start angle=45, end angle=75];
\draw (60:1.5) node{$g'$};
\filldraw[fill=blue!20] (90:1) circle (3pt);
\draw (225:1) .. controls (240:.5) .. (255:1);
\draw (-45:1) .. controls (-30:.5) .. (-15:1);
\end{tikzpicture}
}

\newcommand{\PICABG}
{
\begin{tikzpicture}[baseline=-\the\dimexpr\fontdimen22\textfont2\relax]
\draw (20:1) arc[radius=1cm,start angle=20, end angle=340];
\filldraw (45:1) circle (3pt) node[above right]{$C$};
\filldraw (120:1) circle (3pt) node[above]{$C_2$};
\draw[line width=1pt, loosely dotted] (150:1.2) arc[radius=1.2cm,start angle=150, end angle=180];
\filldraw (210:1) circle (3pt) node[left]{$C_b$};
\filldraw[fill=blue!20] (270:1) circle (3pt);
\draw[line width=1pt, loosely dotted] (285:1.2) arc[radius=1.2cm,start angle=285, end angle=315];
\draw (300:1.5) node{$g$};
\filldraw[fill=blue!20] (330:1) circle (3pt);
\draw (105:1) .. controls (120:.5) .. (135:1);
\draw (195:1) .. controls (210:.5) .. (225:1);
\draw (3,0) +(-160:1) arc[radius=1cm,start angle=-160, end angle=160];
\draw (3,0) +(160:1) node (mynode){};
\draw (20:1) .. controls +(0,-.3) and +(0,-.3) .. (mynode.center);
\draw (3,0) +(-160:1) node (mynode2){};
\draw (-20:1) .. controls +(0,.3) and +(0,.3) .. (mynode2.center);
\filldraw (3,0) +(135:1) circle (3pt) node[above left]{$C'$};
\filldraw (3,0) +(240:1) circle (3pt) node[below]{$C'_2$};
\draw[line width=1pt, loosely dotted] (3,0) +(270:1.2) arc[radius=1.2cm,start angle=270, end angle=300];
\filldraw (3,0) +(-30:1) circle (3pt) node[right]{$C'_{b'}$};
\filldraw[fill=blue!20] (3,0) +(30:1) circle (3pt);
\draw[line width=1pt, loosely dotted] (3,0) +(45:1.2) arc[radius=1.2cm,start angle=45, end angle=75];
\draw (3,0) +(60:1.5) node{$g'$};
\filldraw[fill=blue!20] (3,0) +(90:1) circle (3pt);
\draw (3,0) +(255:1) node (mynode3){};
\draw (3,0) +(240:.5) node (mynode5){};
\draw (3,0) +(225:1) .. controls (mynode5.center) .. (mynode3.center);
\draw (3,0) +(-15:1) node (mynode4){};
\draw (3,0) +(-30:.5) node (mynode6){};
\draw (3,0) +(-45:1) .. controls (mynode6.center) .. (mynode4.center);
\end{tikzpicture}
}

\newcommand{\PICABH}
{
\begin{tikzpicture}[baseline=-\the\dimexpr\fontdimen22\textfont2\relax]
\draw (0,0) circle (1cm);
\filldraw[fill=blue!20] (5:1) circle (3pt);
\draw[line width=1pt, loosely dotted] (10:1.2) arc[radius=1.1cm,start angle=10, end angle=35];
\draw (22.5:1.5) node{$g'$};
\filldraw[fill=blue!20] (40:1) circle (3pt);
\filldraw (70:1) circle (3pt) node[above right]{$C'$};
\filldraw (110:1) circle (3pt) node[above left]{$C$};
\filldraw (155:1) circle (3pt) node[left]{$C_2$};
\draw (140:1) .. controls (145:.5) .. (170:1);
\draw[line width=1pt, loosely dotted] (170:1.2) arc[radius=1.1cm,start angle=170, end angle=190];
\filldraw (205:1) circle (3pt) node[left]{$C_b$};
\draw (190:1) .. controls (205:.5) .. (220:1);
\filldraw[fill=blue!20] (230:1) circle (3pt);
\draw[line width=1pt, loosely dotted] (235:1.2) arc[radius=1.1cm,start angle=235, end angle=260];
\draw (247.5:1.5) node{$g$};
\filldraw[fill=blue!20] (265:1) circle (3pt);
\filldraw (290:1) circle (3pt) node[below]{$C'_2$};
\draw (275:1) .. controls (290:.5) .. (305:1);
\draw[line width=1pt, loosely dotted] (305:1.2) arc[radius=1.1cm,start angle=305, end angle=325];
\filldraw (340:1) circle (3pt) node[right]{$C'_{b'}$};
\draw (325:1) .. controls (340:.5) .. (355:1);
\end{tikzpicture}
}

\newcommand{\PICABI}
{
\begin{tikzpicture}[baseline=-\the\dimexpr\fontdimen22\textfont2\relax]
\draw (0,0) circle (1cm);
\filldraw[fill=blue!20] (-40:1) circle (3pt);
\draw[line width=1pt, loosely dotted] (-35:1.2) arc[radius=1.1cm,start angle=-35, end angle=-8];
\draw (-24:1.5) node{$g$};
\filldraw[fill=blue!20] (-8:1) circle (3pt);
\filldraw[fill=blue!20] (8:1) circle (3pt);
\draw[line width=1pt, loosely dotted] (10:1.2) arc[radius=1.1cm,start angle=8, end angle=35];
\draw (24:1.5) node{$g'$};
\filldraw[fill=blue!20] (40:1) circle (3pt);
\filldraw (70:1) circle (3pt) node[above right]{$C'$};
\filldraw (110:1) circle (3pt) node[above left]{$C$};
\filldraw (155:1) circle (3pt) node[left]{$C_2$};
\draw (140:1) .. controls (145:.5) .. (170:1);
\draw[line width=1pt, loosely dotted] (170:1.2) arc[radius=1.1cm,start angle=170, end angle=190];
\filldraw (205:1) circle (3pt) node[left]{$C_b$};
\draw (190:1) .. controls (205:.5) .. (220:1);
\filldraw (245:1) circle (3pt) node[below left]{$C'_2$};
\draw (230:1) .. controls (245:.5) .. (260:1);
\draw[line width=1pt, loosely dotted] (260:1.2) arc[radius=1.1cm,start angle=260, end angle=285];
\filldraw (295:1) circle (3pt) node[below right]{$C'_{b'}$};
\draw (280:1) .. controls (295:.5) .. (310:1);
\end{tikzpicture}
}

\newcommand{\PICABJ}
{
\begin{tikzpicture}[baseline=-\the\dimexpr\fontdimen22\textfont2\relax]
\draw (0,0) circle (1cm);
\filldraw (0:1) circle (1pt) node[right]{$a$};
\filldraw (30:1) circle (3pt) node[right]{$A$};
\filldraw (60:1) circle (1pt) node[above right]{$b$};
\filldraw (90:1) circle (3pt) node[above]{$B$};
\filldraw (120:1) circle (3pt) node[above left]{$C_2$};
\draw[line width=1pt, loosely dotted] (150:1.2) arc[radius=1.2cm,start angle=150, end angle=180];
\filldraw (210:1) circle (3pt) node[left]{$C_b$};
\filldraw[fill=blue!20] (270:1) circle (3pt);
\draw[line width=1pt, loosely dotted] (285:1.2) arc[radius=1.2cm,start angle=285, end angle=315];
\draw (300:1.5) node{$g$};
\filldraw[fill=blue!20] (330:1) circle (3pt);
\draw (105:1) .. controls (120:.5) .. (135:1);
\draw (195:1) .. controls (210:.5) .. (225:1);
\end{tikzpicture}
}

\newcommand{\PICABK}
{
\begin{tikzpicture}[baseline=-\the\dimexpr\fontdimen22\textfont2\relax]
\draw (0,0) circle (1cm);
\filldraw (30:1) circle (3pt) node[right]{$A$};
\filldraw (90:1) circle (3pt) node[above]{$B$};
\filldraw (120:1) circle (3pt) node[above left]{$C_2$};
\draw[line width=1pt, loosely dotted] (150:1.2) arc[radius=1.2cm,start angle=150, end angle=180];
\filldraw (210:1) circle (3pt) node[left]{$C_b$};
\filldraw[fill=blue!20] (270:1) circle (3pt);
\draw[line width=1pt, loosely dotted] (285:1.2) arc[radius=1.2cm,start angle=285, end angle=315];
\draw (300:1.5) node{$g$};
\filldraw[fill=blue!20] (330:1) circle (3pt);
\draw (0:1) .. controls (30:.5) .. (60:1);
\draw (105:1) .. controls (120:.5) .. (135:1);
\draw (195:1) .. controls (210:.5) .. (225:1);
\end{tikzpicture}
}

\newcommand{\PICABL}
{
\begin{tikzpicture}[baseline=-\the\dimexpr\fontdimen22\textfont2\relax]
\draw (0,0) circle (1cm);
\filldraw (40:1) circle (3pt) node[above right]{$B$};
\filldraw (80:1) circle (3pt) node[above]{$A$};
\filldraw (120:1) circle (3pt) node[above left]{$C_2$};
\draw[line width=1pt, loosely dotted] (150:1.2) arc[radius=1.2cm,start angle=150, end angle=180];
\filldraw (210:1) circle (3pt) node[left]{$C_b$};
\filldraw[fill=blue!20] (270:1) circle (3pt);
\draw[line width=1pt, loosely dotted] (285:1.2) arc[radius=1.2cm,start angle=285, end angle=315];
\draw (300:1.5) node{$g$};
\filldraw[fill=blue!20] (330:1) circle (3pt);
\draw (95:1) .. controls (80:.5) .. (65:1);
\draw (105:1) .. controls (120:.5) .. (135:1);
\draw (195:1) .. controls (210:.5) .. (225:1);
\end{tikzpicture}
}

\newcommand{\PICABM}
{
\begin{tikzpicture}[baseline=-\the\dimexpr\fontdimen22\textfont2\relax]
\draw (0,0) circle (1cm);
\filldraw (0:1) circle (1pt) node[right]{$a$};
\filldraw (20:1) circle (3pt) node[right]{$A$};
\filldraw (60:1) circle (1pt) node[above right]{$b$};
\filldraw (80:1) circle (3pt) node[above]{$B$};
\filldraw (120:1) circle (3pt) node[above]{$C_3$};
\draw[line width=1pt, loosely dotted] (150:1.2) arc[radius=1.2cm,start angle=150, end angle=180];
\filldraw (210:1) circle (3pt) node[left]{$C_b$};
\filldraw[fill=blue!20] (270:1) circle (3pt);
\draw[line width=1pt, loosely dotted] (285:1.2) arc[radius=1.2cm,start angle=285, end angle=315];
\draw (300:1.5) node{$g$};
\filldraw[fill=blue!20] (330:1) circle (3pt);
\draw (95:1) .. controls (70:.5) .. (45:1);
\draw (105:1) .. controls (120:.5) .. (135:1);
\draw (195:1) .. controls (210:.5) .. (225:1);
\end{tikzpicture}
}

\newcommand{\PICABN}
{
\begin{tikzpicture}[baseline=-\the\dimexpr\fontdimen22\textfont2\relax]
\draw (0,0) circle (1cm);
\filldraw (20:1) circle (3pt) node[right]{$A$};
\filldraw (80:1) circle (3pt) node[above]{$B$};
\filldraw (120:1) circle (3pt) node[above]{$C_3$};
\draw[line width=1pt, loosely dotted] (150:1.2) arc[radius=1.2cm,start angle=150, end angle=180];
\filldraw (210:1) circle (3pt) node[left]{$C_b$};
\filldraw[fill=blue!20] (270:1) circle (3pt);
\draw[line width=1pt, loosely dotted] (285:1.2) arc[radius=1.2cm,start angle=285, end angle=315];
\draw (300:1.5) node{$g$};
\filldraw[fill=blue!20] (330:1) circle (3pt);
\draw[line width=2pt,red] (0:1) .. controls (30:.5) .. (60:1);
\draw (95:1) .. controls (70:.5) .. (45:1);
\draw (105:1) .. controls (120:.5) .. (135:1);
\draw (195:1) .. controls (210:.5) .. (225:1);
\end{tikzpicture}
}

\newcommand{\PICABO}
{
\begin{tikzpicture}[baseline=-\the\dimexpr\fontdimen22\textfont2\relax]
\draw (0,0) circle (1cm);
\filldraw (20:1) circle (3pt) node[right]{$A$};
\filldraw (-20:1) circle (3pt) node[right]{$B$};
\filldraw (120:1) circle (3pt) node[above]{$C_3$};
\draw[line width=1pt, loosely dotted] (140:1.2) arc[radius=1.2cm,start angle=140, end angle=170];
\filldraw (190:1) circle (3pt) node[left]{$C_b$};
\filldraw[fill=blue!20] (250:1) circle (3pt);
\draw[line width=1pt, loosely dotted] (265:1.2) arc[radius=1.2cm,start angle=265, end angle=295];
\draw (280:1.5) node{$g$};
\filldraw[fill=blue!20] (310:1) circle (3pt);
\draw[line width=2pt,red] (80:1) .. controls (62.5:.5) .. (45:1);
\draw (0:1) .. controls (30:.5) .. (60:1);
\draw (105:1) .. controls (120:.5) .. (135:1);
\draw (175:1) .. controls (190:.5) .. (205:1);
\end{tikzpicture}
}

\newcommand{\PICABP}
{
\begin{tikzpicture}[baseline=-\the\dimexpr\fontdimen22\textfont2\relax]
\draw (0,0) circle (1cm);
\filldraw (110:1) circle (3pt) node[above]{$A$};
\filldraw (0:1) circle (3pt) node[right]{$B$};
\filldraw (140:1) circle (3pt) node[above left]{$C_3$};
\draw[line width=1pt, loosely dotted] (170:1.2) arc[radius=1.2cm,start angle=170, end angle=200];
\filldraw (230:1) circle (3pt) node[left]{$C_b$};
\filldraw[fill=blue!20] (270:1) circle (3pt);
\draw[line width=1pt, loosely dotted] (285:1.2) arc[radius=1.2cm,start angle=285, end angle=315];
\draw (300:1.5) node{$g$};
\filldraw[fill=blue!20] (330:1) circle (3pt);
\draw (80:1) .. controls (62.5:.5) .. (45:1);
\draw (95:1) .. controls (77.5:.5) .. (60:1);
\draw (125:1) .. controls (140:.5) .. (155:1);
\draw (215:1) .. controls (230:.5) .. (245:1);
\end{tikzpicture}
}

\newcommand{\PICABQ}
{
\begin{tikzpicture}[baseline=-\the\dimexpr\fontdimen22\textfont2\relax]
\draw (0,0) circle (1cm);
\filldraw (110:1) circle (3pt) node[above]{$A$};
\filldraw (0:1) circle (3pt) node[right]{$B$};
\filldraw (140:1) circle (3pt) node[above left]{$C_3$};
\draw[line width=1pt, loosely dotted] (170:1.2) arc[radius=1.2cm,start angle=170, end angle=200];
\filldraw (230:1) circle (3pt) node[left]{$C_b$};
\filldraw[fill=blue!20] (270:1) circle (3pt);
\draw[line width=1pt, loosely dotted] (285:1.2) arc[radius=1.2cm,start angle=285, end angle=315];
\draw (300:1.5) node{$g$};
\filldraw[fill=blue!20] (330:1) circle (3pt);
\filldraw[fill=blue!20] (45:1) circle (3pt);
\draw (125:1) .. controls (140:.5) .. (155:1);
\draw (215:1) .. controls (230:.5) .. (245:1);
\end{tikzpicture}
}

\newcommand{\PICABR}
{
\begin{tikzpicture}[baseline=-\the\dimexpr\fontdimen22\textfont2\relax]
\draw (0,0) circle (1cm);
\filldraw (80:1) circle (3pt) node[above]{$A$};
\filldraw (40:1) circle (3pt) node[above right]{$B$};
\filldraw (120:1) circle (3pt) node[above]{$C_3$};
\draw[line width=1pt, loosely dotted] (150:1.2) arc[radius=1.2cm,start angle=150, end angle=180];
\filldraw (210:1) circle (3pt) node[left]{$C_b$};
\filldraw[fill=blue!20] (270:1) circle (3pt);
\draw[line width=1pt, loosely dotted] (285:1.2) arc[radius=1.2cm,start angle=285, end angle=315];
\draw (300:1.5) node{$g$};
\filldraw[fill=blue!20] (330:1) circle (3pt);
\filldraw[fill=blue!20] (0:1) circle (3pt);
\draw (105:1) .. controls (120:.5) .. (135:1);
\draw (195:1) .. controls (210:.5) .. (225:1);
\end{tikzpicture}
}

\newcommand{\PICABS}
{
\begin{tikzpicture}[baseline=-\the\dimexpr\fontdimen22\textfont2\relax]
\draw (0,0) circle (1cm);
\filldraw (0:1) circle (3pt) node[right]{$C_1$};
\filldraw (45:1) circle (3pt) node[above right]{$C_2$};
\filldraw (90:1) circle (3pt) node[above]{$C_3$};
\draw[line width=1pt, loosely dotted] (115:1.2) arc[radius=1.2cm,start angle=115, end angle=155];
\filldraw (180:1) circle (3pt) node[left]{$C_b$};
\draw (205:1) .. controls (225:.5) .. (245:1);
\draw (225:1) .. controls (245:.5) .. (265:1);
\draw (285:1) .. controls (305:.5) .. (325:1);
\draw (305:1) .. controls (325:.5) .. (345:1);
\draw[line width=1pt, loosely dotted] (265:1.2) arc[radius=1.2cm,start angle=265, end angle=285];
\draw (30:1) .. controls (45:.5) .. (60:1);
\draw (75:1) .. controls (90:.5) .. (105:1);
\draw (165:1) .. controls (180:.5) .. (195:1);
\end{tikzpicture}
}


%% file: environments.tex
\usepackage{amsthm}

\swapnumbers
\theoremstyle{definition}
\newtheorem{theorem}{Theorem}[section]

\newtheorem{lemma}[theorem]{Lemma}

\newtheorem{example}[theorem]{Example}
\newtheorem{remark}[theorem]{Remark}
\newtheorem{definition}[theorem]{Definition}